\documentclass[11pt]{article}
\usepackage[paperwidth=8.5in, paperheight=11in, margin=.75 in]{geometry}
\usepackage{amsmath, amsthm,amssymb,amsfonts}
\usepackage{mathrsfs}
\usepackage{verbatim}
\usepackage{multicol}
\usepackage{graphicx}

\def\mA{\mathbb{A}}

\def\mM{\mathbb{M}}

\def\mP{\mathbb{P}}

\def\mS{\mathbb{S}}
\def\mW{\mathbb{W}}
\def\mX{\mathbb{X}}
\def\mY{\mathbb{Y}}
\def\mZ{\mathbb{Z}}

\def\cE{\fs{E}}
\def\cF{\fs{F}}
\def\cI{\mc{I}}

\def\bE{\mathbf{E}}
\def\bJ{\mathbf{J}}
\def\bU{\mathbf{U}}
\def\bV{\mathbf{V}}

\def\bk{\mathbf{k}}
\def\bx{\mathbf{x}}
\def\bu{\mathbf{u}}

\def\btau{\boldsymbol\tau}

\def\FT{\mc{T}}
\def\FS{\mc{S}}

\def\d{\partial}

\usepackage{xcolor,framed,fancyvrb}

\definecolor{lightgray}{gray}{0.75}
\colorlet{shadecolor}{gray!50}

\newtheorem{thm}{Theorem}

\newtheorem{lem}[thm]{Lemma}

\theoremstyle{definition}

\newtheorem{exper}{Experiment}

\theoremstyle{remark}
\newtheorem*{rem}{Remark}

\numberwithin{equation}{section}

\newcommand{\ep}{\epsilon}

\newcommand{\w}{\omega}

\newcommand{\Dv}{{\bf D}}
\newcommand{\Ev}{{\bf E}}

\newcommand{\Jv}{{\bf J}}

\newcommand{\Lv}{{\bf L}}

\newcommand{\phiv}{\boldsymbol\phi}

\newcommand{\mr}[1]{{\mathrm{#1}}}
\newcommand{\mc}[1]{{\mathcal{#1}}}

\newcommand{\scurl}{\text{curl}}
\newcommand{\vcurl}{\text{\bf{curl}}}

\newcommand{\od}[1]{{\frac{\partial}{\partial #1}}}
\newcommand{\ods}[1]{{\frac{\partial^2}{\partial {#1}^2}}}
\newcommand{\dd}[2]{{\partial_{#1} #2}}

\newcommand{\pd}[2]{{\partial^2_{#1} #2}}

\newcommand{\mb}[1]{{\mathbf{#1}}}
\newcommand{\fs}[1]{\mathscr{#1}}
\newcommand{\ds}{\displaystyle}

\newcommand{\e}{\mathrm{e}}
\renewcommand{\i}{\mathrm{i}}

\newcommand{\dx}{\Delta x}
\newcommand{\dy}{\Delta y}
\newcommand{\dt}{\Delta t}

\title{A DISPERSION MINIMIZED MIMETIC METHOD FOR A COLD PLASMA MODEL}

\author{Vrushali A. Bokil$^1$, Vitaliy Gyrya$^2$, and Duncan A. McGregor$^1$}

\smallskip

\begin{document}
\maketitle

\begin{center}
$^1$Oregon State University\\
  Department of Mathematics, Corvallis, OR 97331\\
  e-mail: \{bokilv,mcgregod\}@math.oregonstate.edu\\
  $^2$Los Alamos National Laboratory \\
  T-5, Applied Mathematics and Plasma Physics, Los Alamos, NM\\
  e-mail: vitaliy$\_$gyrya@lanl.gov
\end{center}

\noindent {\bf Keywords}: Maxwell's Equations, Mimetic Methods, Cold Plasma, Exponential Time Differencing, Numerical Dispersion Relation, Yee Scheme.

\medskip

\noindent {\bf MSC (2000)}: 65M06, 65M60, 78M20, 78M10

\medskip

\noindent {\bf Abstract}: In this paper we consider the lowest edge-based mimetic finite difference (MFD) discretization in space for Maxwell's equations in cold plasma on rectangular meshes. The method uses a generalized form of mass lumping that, on one hand, eliminates a need for linear solves at every iteration while, on the other hand, retains a set of free parameters of the MFD discretization. We perform an optimization procedure, called m-adaptation, that identified a set of free parameters that lead to the smallest numerical dispersion.
The choice of the time stepping proved to be critical for successful optimization. Using exponential time differencing we were able to reduce the numerical dispersion error from second to fourth order of accuracy in mesh size. It was not possible to achieve this order of magnitude reduction in numerical dispersion error using the standard leapfrog time stepping. Numerical simulations independently verify our theoretical findings.

\section{Introduction}
\label{intro}

A variety of numerical techniques are available in the literature for the simulation of Electromagnetic (EM) wave propagation in
linear dispersive media, the gold standard being the Yee Finite difference Time Domain (FDTD) method \cite{kashiwa1990treatment}.
EM wave propagation in a medium is modeled by Maxwell's Equations, a vector system of Partial Differential Equations (PDEs), that govern the evolution of the EM field, along with appropriate constitutive equations for the response of the medium to the EM field. The Yee scheme is a FDTD method that simultaneously discretizes Maxwell's equations along with the constitutive laws for the medium to produce a second order accurate discretization in space and time. However, the second order numerical dispersion errors that arise in the discrete solution
obtained using the Yee scheme can lead to large errors over long time integration on electrically large domains. Thus, the construction of numerical methods with high order numerical dispersion errors for linear dispersive media, which is the goal of this paper, is crucial to the accurate simulation of EM waves in such media.

In this paper, we consider the simulation of EM waves in a cold isotropic plasma,
a type of linear dispersive medium. The model for cold plasma is based on the Auxiliary Differential Equation (ADE) approach
in which an Ordinary Differential Equation (ODE) for the evolution of the time derivative of the macroscopic polarization (polarization current density) is appended to Maxwell's equations to produce a hybrid PDE-ODE system.
The evolution ODE for the polarization current density models the averaged response of the material to the electromagnetic field.
We present the construction of a dispersion minimized numerical method for Maxwell's equations in a cold plasma by performing a novel optimization procedure, called {\it m-adaptation},
on a family of numerical methods for the cold plasma model based on the Mimetic Finite Difference (MFD)
method in space and Exponential Time Differencing (ETD) in time.

MFD methods are a flexible family of methods that are based on general polygonal and polyhedral meshes, see \cite{Lipnikov-Manzini-Shashkov:2014} for a comprehensive review.
The word mimetic indicates the fact that they mimic/preserve in the discrete settings some properties of the continuous equations.
The MFD construction is generally non-unique and leads to a parameterized family of methods
with equivalent properties, such as stencil size and base convergence rate among others.
Many of the classical discretizations are contained within the MFD family,
e.g. Yee scheme on rectangular elements.
The number of parameters characterizing the scheme grows rapidly with the dimension, the number of vertices in a polygonal element, and the order of the discretization.
The parameterized nature of the family of MFD methods presents an opportunity for optimization
for some desired properties.
M-adaptation, as introduced in \cite{da2014mimetic,gyrya2014m},
is the process of selecting an optimal member of the family of mimetic schemes for a
selected optimization criteria (in this case, minimization of numerical dispersion).
In \cite{gyrya2014m} other optimization criteria were analysed.

We have previously considered the problem of optimizing numerical dispersion error in models of electromagnetic wave propagation in free space \cite{bokil2015dispersion}. In this earlier work, we started with the parameterized family of MFD schemes that all have second (base) order of numerical dispersion error on rectangular meshes.
Through m-adaptation we produced a method that has fourth order numerical dispersion error.
The present paper extends our prior work to a cold plasma model which, as discussed above, appends an additional evolution equation for the polarization current density to Maxwell's equations. The extension of m-adaptation to the cold plasma model proved to be non-trivial, as illustrated by our first failed attempt presented in the appendix section of this paper.
It turns out, that for linear dispersive media, the choice of the time discretization scheme is critical for the m-adaptation technique to produce a high order method.
The Leapfrog time differencing method,
which we previously employed in \cite{bokil2015dispersion}
for the case of EM propagation in free space,
does not allow m-adaptation to produce a higher order method for the cold plasma model.
Instead, replacing the Leapfrog time differencing with ETD
allows for successful optimization over the family of fully discrete MFD schemes.
The optimal Exponential Time Mimetic Finite Difference (ETMFD) method, produced by m-adaptation, has a fourth order numerical dispersion error as compared to the
(base) second order error for the rest of the ETMFD schemes in this family on rectangular meshes.

The ETD method was originally introduced in computational electromagnetism as a scheme for handling stiff problems, such as computing the
electric and magnetic fields in a box surrounded by perfectly matched layers \cite{cox2002exponential}. For these problems, explicit time-stepping, for example the Leapfrog time differencing method, requires an extremely small time step in order to be stable. On the other hand implicit schemes that are unconditionally stable can be costly to implement in three dimensions. ETD involves an exact integration of some of the lower order linear terms in the governing equations, and higher order accuracy can be obtained by using a higher order discretization of the resulting integral terms. However, the ETD approach has been shown to offer no major advantages over the time averaging of the lower order linear terms in the Yee scheme, for alleviating stiffness. In some cases, ETD may be less efficient by necessitating smaller step sizes \cite{petro1997etd}.

We would like to emphasize that the reason for the choice of ETD in our work is not for handling stiffness, but rather that it is a good candidate for a time discretization method which allows for successful optimization in the m-adaptation technique.
In contrast to other numerical methods for the cold plasma model that use ETD discretization only for the equation of polarization current density \cite{cummer1997fdtdcip}, our ETMFD method is a discretization of a hyrbrid PDE-ODE system modeling the evolution of the polarization current density and electric field forced by spatial derivatives of field variables.
%

The outline of the paper is as follows.
In Section~\ref{model} we present first and second order PDE models
for cold plasma, and the corresponding weak formulations.
The lowest order edge based MFD discretization for the electric field and current density on rectangular meshes is presented in Section~\ref{spacediscrete}, while the exponential time difference discretization
for scalar and vector equations is presented in Section~\ref{timediscrete}.
In Section~\ref{fulldiscrete} we present a fully discrete ETMFD family of discretizations
that employs a generalized form of mass-lumping
to produce a fully explicit scheme to avoid the need for linear solves at every time step.
In Section \ref{madapt}
we derive the numerical dispersion relation for this family of discrete schemes and
choose the set of MFD parameters that produces a method with the lowest numerical dispersion error.
This optimization requires numerical dispersion properties derived in Sections~\ref{spacediscrete} and \ref{timediscrete}.
In Section \ref{num} we present results of numerical simulations that
independently validate our theoretical results.
In Section \ref{conclude} we present some concluding remarks.
Finally, in Appendix \ref{appndx}, we demonstrate that a straightforward extension of our ideas
developed in \cite{bokil2015dispersion}
does not allow the m-adaptation process to reduce numerical dispersion error, thus,
demonstrating the novelty of the ETMFD method presented in this paper.

\section{Maxwell's Equations in a Cold Plasma}
\label{model}

The Cold  Plasma (CP) model is a special case of the Lorentz model \cite{kashiwa1990treatment}
 governing the evolution of electromagnetic waves in partially ionized gases
without magnetization effects.
It consists of Maxwell's equations along with evolution equations
for the time derivative of the macroscopic electric polarization field. This time derivative is called the polarization current density.
Suppose $\Omega \subset \mathbb{R}^2$ and $T>0$.
Maxwell's equations governing the evolution of wave propagation on $\Omega \times [0,T]$
relate the electric field intensity $\textbf{E}$, and the magnetic flux density $B$ as
\begin{subequations}\label{eq:Maxwell equations}
\label{eq:max}
\begin{align}
    & \frac{\partial}{\partial t}\textbf{E}
    =
    c_0^2\ \vcurl \ B-\frac{1}{\epsilon_0}\textbf{J}, \\
    & \frac{\partial}{\partial t}B
    =
    - \scurl \ \textbf{E},
\end{align}
\end{subequations}
where $c_0$, and $\epsilon_0$, are the speed of light,
and the electric permittivity of free space, respectively.
The vector $\textbf{J}$ is the polarization current density and is modeled by the evolution equation
\begin{equation} \label{eq:evolutionJ}
    \displaystyle \od t \mb J + \w_i\mb J = \ep_0\w_P^2\mb E.
\end{equation}
Here $\w_i$ is the ion collision frequency and $\w_P$ is the plasma frequency.
For a vector field ${\bf f} = (\mb{f}_x,\mb{f}_y)^T$ and
for a scalar field $f$ we define
the scalar ($\scurl$) and
vector ($\vcurl$) curl operators as follows,
\[
    \scurl({\bf f}) := \od{x}{\bf f}_y-\od{y}{\bf f}_x,
    \qquad
    \vcurl(f):=\left(\od{y}f, -\od{x}f\right)^T.
\]
All the fields in the system \eqref{eq:max}-\eqref{eq:evolutionJ} are functions of
position $\mb{x} = (x,y)^T$ and time $t \in [0,T]$.
We also assume perfect electrical conductor (PEC) boundary conditions.
\begin{equation}
\textbf{E} \times \textbf{n} = 0, \ \text{on} \ \partial\Omega \times [0,T],
\end{equation}
where $\textbf{n}$ is a unit outward vector to the boundary $\Omega$.
The equations \eqref{eq:Maxwell equations} and \eqref{eq:evolutionJ}
are subject to appropriate initial conditions.

The first order equations (\ref{eq:Maxwell equations}-\ref{eq:evolutionJ})
can be written in an equivalent second order formulation,
which we call the {\bf Maxwell-CP Model},
by eliminating the magnetic flux density field $B$ as
\begin{align}\label{eq:MECP}
\left\{
\begin{array}{ll}
    \ds \ods t \mb E = -\frac{1}{\ep_0}\od t \mb J -c_0^2\ \vcurl \ \scurl \ \mb E,\qquad
    & \text{in} \hspace{4mm}\Omega \times (0,T], \\[1.5ex]
    \ds \od t \mb J = -\w_i\mb J + \ep_0\w_p^2\mb E,
    & \text{in} \hspace{4mm}\Omega \times (0,T], \\[1.5ex]
    \mb E\times\mb n =0, & \text{on}  \ \partial\Omega \times (0,T],
\end{array}
\right .
\end{align}
along with appropriate initial conditions.
We will construct a MFD discretization based on this second order formulation.
%

\subsection{Variational Formulation}

The MFD discretization, just like a finite element formulation, will be constructed
based on the weak form of \eqref{eq:MECP}. To this end, we consider the Sobolev spaces
\begin{align*}
    \Lv^2(\Omega) &= [L^2(\Omega)]^2,\\
    \mb H(\scurl, \Omega ) &= \{ {\bf v}\in\Lv^2(\Omega): \scurl \ {\bf v} \in L^2(\Omega)\},\\
    \mb H_0(\scurl, \Omega ) &= \{ {\bf v}\in  \mb H(\scurl, \Omega ), \mb v\times\mb n =0, \ \text{on}  \ \partial\Omega\}.
\end{align*}
The weak formulation of \eqref{eq:MECP} is obtained in a standard way.
Multiply the first, and the second equation, in \eqref{eq:MECP}
by test functions, $\phiv \in  \mb H_0(\scurl, \Omega )$, and $\psi \in  \Lv^2(\Omega)$, respectively, and integrate over the domain $\Omega$.
The weak formulation reads\\
\emph{Find}
$\Ev \in C^2([0,T];\mb H_0(\scurl, \Omega ) )$, \emph{and} $\Jv \in C^1([0,T]; \Lv^2(\Omega) )$,
\emph{subject to appropriate initial conditions, such that for all} $\phiv \in  \mb H_0(\scurl, \Omega ) $ and $\psi \in \Lv^2(\Omega)$ we have
\begin{align}
    \label{eq:varevecwave}
    \left[{\Ev_{tt}}, \phiv\right]_{E}
    +
    c_0^2\left[ (\scurl \ \Ev),  (\scurl \ \phiv) \right]_F
    +
    \frac{1}{\epsilon_0}
    \left[\Jv_t,  \phiv\right]_{E}
    =
    0,
    \\
    \label{eq:varevecwave2}
    \left[\Jv_t,  \psi\right]_E
    +
    \w_i
    \left[\Jv, \psi\right]_E
    -
    \epsilon_0\w_p^2
    \left[\Ev, \psi\right]_E
    =
    0,
\end{align}
where the bilinear forms are defined as follows:
\begin{equation}
  \label{eq:def:bilin forms E n F}
  [\Jv,\Ev]_E := \int_\Omega \Jv \cdot \Ev \ d\Omega, \qquad
  [J,E]_F := \int_\Omega J\  E \ d\Omega.
\end{equation}
Here $(\Jv,\Ev)$ are vector functions, and $(J,E)$ are scalar functions.

\subsection{Dispersion Relation}
\label{sec:dispersion relation}

In this paper our aim is to construct a numerical method for the Maxwell-CP model \eqref{eq:MECP} that is the optimal method
chosen from a family of schemes by minimizing for numerical dispersion error.
Thus, in this section,
we present a brief overview of continuous and numerical dispersion relations and
their connections to symbols of differential operators.

Given a plane wave
\begin{equation}\label{eq:planewave}
  \bu(t,\bx):=\e^{\i(\mb k\cdot\mb x-\w t)} \bu_0,
\end{equation}
a continuous or discrete
\emph{dispersion relation} is a relation between the frequency $\w$
and the wave vector $\mb k$, under which $\bu$, or its restriction to a discrete grid, is a solution of a continuum PDE, or its discrete approximation, respectively.

Consider an abstract linear equation
\begin{equation}\label{eq:abstract linear equation}
  \mathcal{L}_t\{\bu\}
  =
  \mathcal{L}_\bx\{\bu\},
\end{equation}
where $\mathcal{L}_t$ and $\mathcal{L}_x$
are linear operators corresponding to time and space, respectively,
and $\bu$ is either a continuous plane wave \eqref{eq:planewave} or its discrete representation.
For example, $\mathcal{L}_t$ and $\mathcal{L}_\bx$ could be the
continuous differential operators
$\mathcal{L}_t\{\bu\} = \frac{\d^2}{\d t^2}\bu$
and
$\mathcal{L}_\bx\{\bu\} = \triangle \bu$
or their discrete approximations.

As it turns out, for all linear operators $\mathcal{L}_t$ and $\mathcal{L}_x$
(continuous or discrete)
considered in this paper
the plane wave \eqref{eq:planewave} is a generalized eigenfunction, i.e.
\begin{equation}\label{eq:eigen function}
  \mc L_t\{\bu\} = \FT(\w) \bu
  \qquad\text{and}\qquad
  \mc L_x\{\bu\} = \FS(\bk) \bu,
\end{equation}
where $\FT(\w)$ and $\FS(\bk)$ are square matrices acting on $\bu_0$, i.e.
\[
    \FT(\w ) \bu = \e^{\i(\bk\cdot\bx-\w t)}\ \FT(\w )\bu_0
    \qquad\text{and}\qquad
    \FS(\bk) \bu = \e^{\i(\bk\cdot\bx-\w t)}\ \FS(\bk)\bu_0.
\]
If $\FT(\w )$ and $\FS(\bk)$ were scalars, they would be eigenvalues and
$\bu$ would be the eigenfunction of $\mathcal{L}_t$ and $\mathcal{L}_x$.
Since, in general, they are not scalars
we refer to $\bu$ as a generalized eigenfunction and
call $\FT(\w )$ and $\FS(\bk)$ - symbols of linear operators
$\mathcal{L}_t$ and $\mathcal{L}_x$, respectively.

Substituting \eqref{eq:eigen function} into \eqref{eq:abstract linear equation}
and cancelling the exponential terms $\e^{\i(\bk\cdot\bx-\w t)}$ on both sides
we obtain a dispersion relation written in terms of symbols of linear operators
\begin{equation}\label{eq:dispersion through symbols}
  \FT(\w)\bu_0 = \FS(\bk)\bu_0.
\end{equation}

In this work we are primarily concerned with \emph{dispersion error}, which can be defined as the absolute value of the difference between the frequency $\omega(\mb k)$, solution to the continuous dispersion relation, and its discrete counterpart, $\omega_{\Delta t,h}(\mb k)$, solution to the discrete dispersion relation. There are other ways in which the dispersion error can be defined as discussed in Section \ref{num}.  Dispersion error is the result of frequency dependent speed of propagation of plane waves in the discretized grid regardless of whether the continuum solution has such frequency dependent propagation or not. In particular, the speed of propagation of waves in the discrete grid always differs from that in the continuum case and is commonly observed as non-physical oscillations in discrete solutions.
Thus, if $\FT_{\dt}$ is a discrete approximation of $\FT$, and $\FS_h$ is a discrete approximation of $\FS$ then we have
\begin{align}
\FT_{\dt} - \FS_h = \FT-\FS + \mc O(h^\alpha).
\end{align}
Where $\dt>0, h >0$ are mesh resolution parameters. It will be our goal to find $\FT_{\dt}$ and $\FS_h$ so that $\alpha$ is as large as possible reducing the discrepancy in wave speed.

\section{Mimetic Finite Difference Discretization in Space}
\label{spacediscrete}

A mimetic finite difference discretization of the continuous variational formulation
(\ref{eq:varevecwave}-\ref{eq:varevecwave2})
has the form
\begin{align}
    \label{eq: disc wave 1}
    \left[{(\Ev_h)_{tt}}, \phiv_h\right]_\cE
    +
    c_0^2\left[ (\scurl_h \ \Ev_h),  (\scurl_h \ \phiv_h) \right]_\cF
    +
    \frac{1}{\epsilon_0}
    \left[(\Jv_h)_t,  \phiv_h\right]_\cE
    =
    0,
    \\
    \label{eq: disc wave 2}
    \left[(\Jv_h)_t,  \psi_h\right]_\cE
    +
    \w_i
    \left[\Jv_h, \psi_h\right]_\cE
    -
    \epsilon_0\w_p^2
    \left[\Ev_h, \psi_h\right]_\cE
    =
    0.
\end{align}
Here $\Ev_h$, and $\Jv_h$, are discrete approximations of the solutions $\Ev$, and $\Jv$,
respectively;
$\phiv_h$ and $\psi_h$ are discrete test functions;
$\scurl_h$ is a discrete linear operator approximating  its continuous counterpart
$\scurl$.
The bilinear forms
$\left[\cdot, \cdot\right]_\cE$ and
$\left[ \cdot,  \cdot \right]_\cF$
are discrete approximations of the continuous bilinear forms
$\left[\cdot, \cdot\right]_E$ and
$\left[ \cdot,  \cdot \right]_F$
defined in  \eqref{eq:def:bilin forms E n F}.
We will now make all of the above precise.


\subsection{Discrete Spaces and Interpolation}

Let $\fs T$ be a polygonal partitioning (mesh) of the domain $\Omega$.
Let $\fs E$, and $\fs F$, be the set of all edges $e$, and faces (cells) $f$, respectively, of the mesh $\fs T$.
In the discrete form every function will be represented in terms of a finite number of values
called degrees of freedom (DoF) assembled into a vector
(e.g. $\Ev_h$, $\Jv_h$)
assuming some ordering of these DoF.
Each DoF will be associated either to an edge or to a face/element.
The DoF of scalar functions (e.g. $\scurl\, \Ev$)
will be associated with faces only (one DoF per face)
and can be interpreted as an average value of the function over the face/cell.
The DoF of vector functions (e.g. $\Ev$ and $\Jv$)
will be associated with edges only (one DoF per edge)
and can be interpreted as an average value of the tangential component of the vector function
along the edge.
See Figure~\ref{fig:DoF} for illustration.

\begin{figure}[h!]
  \centering
  \includegraphics[width=.5\textwidth]{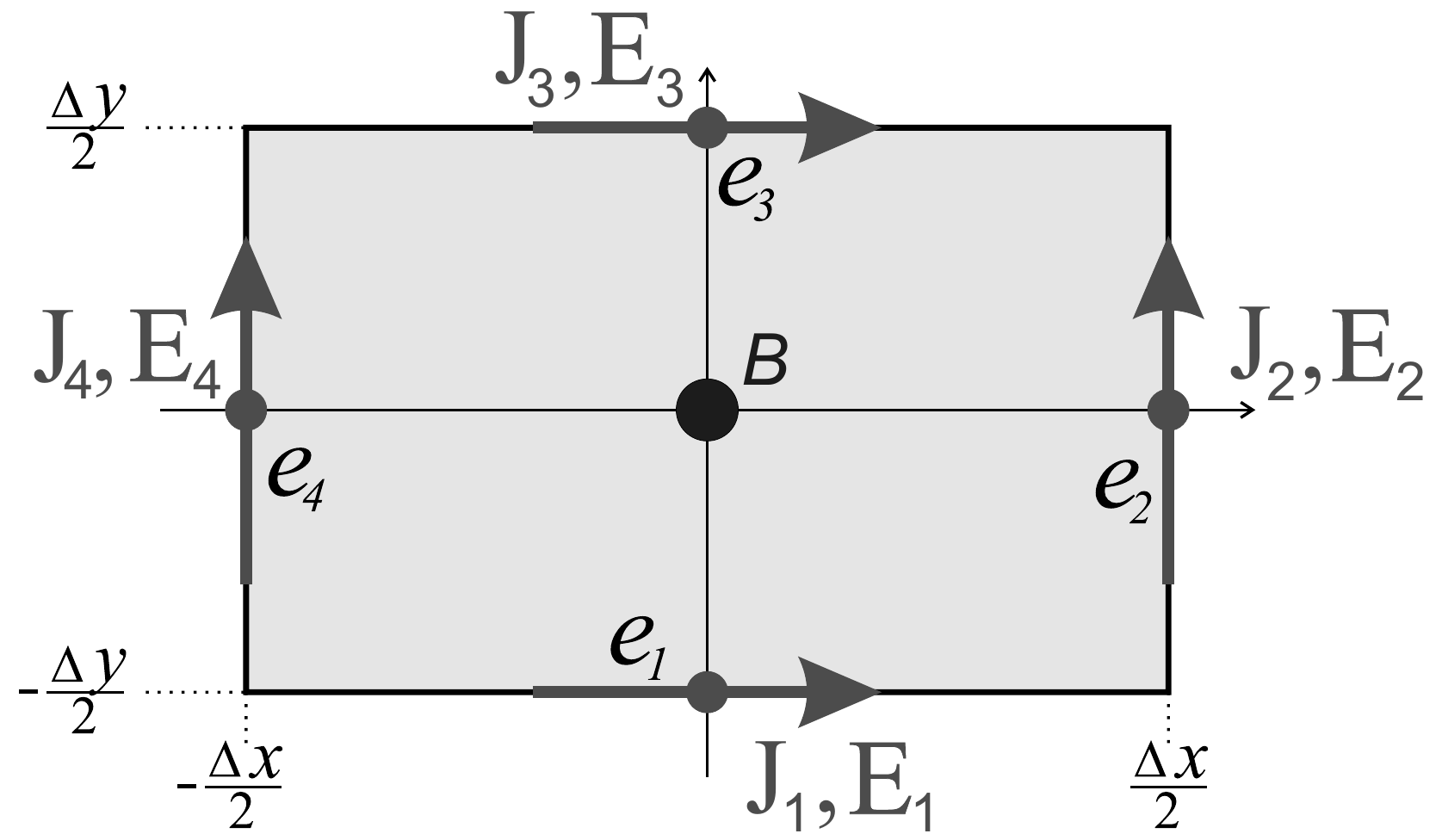}\\
  \caption{Illustration of DoF for vector functions, marked as $\Ev$,
  and scalar functions, marked as $B$ on a rectangular element centered at the origin.}
  \label{fig:DoF}
\end{figure}

We will denote the discrete space corresponding to vector functions as $\fs E_h$
and the discrete space corresponding to scalar functions as $\fs F_h$.
We define interpolation operators $\cI^{\fs E_h}$, and $\cI^{\fs F_h}$,
as linear operators that for each function in
$\mb H(\scurl,\Omega)$, and $L^2$, assigns a vector of DoF in $\fs E_h$, and $\fs F_h$,
respectively.

The interpolation operator  $\cI^{\fs E_h}$, and the vector of DoF for vector functions are defined as
\begin{align}
    &\cI^{\fs E_h}:\mb H(\scurl,\Omega)\to \fs E_h, \qquad
    \bE_e := \displaystyle\frac{1}{|e|}\int_e \bE\cdot\boldsymbol\tau_e \ de,\label{eq:edof}\\
    &\cI^{\fs E_h}(\bE)
    =
    \bE_h
    :=
    \begin{pmatrix}
        \bE_e
    \end{pmatrix}_{e\in\fs E},
\end{align}
where $\boldsymbol\tau_e$ is a counter-clockwise tangent to the edge $e$. The interpolation operator  $\cI^{\fs F_h}$ for scalar functions, and the vector of DoF for scalar functions are defined as
\begin{align}
    &\cI^{\fs F_h}:L^2\to\fs F_h, \qquad
    B_f:= \displaystyle\frac{1}{|f|}\int_f B \ df,\\
    &\cI^{\fs F_h}(B)
    =
    B_h
    :=
    \begin{pmatrix}
        B_f
    \end{pmatrix}_{f\in\fs F}.
\end{align}

\subsection{Discrete Inner Products and the Adjoint Curl}

Discrete bilinear forms
$\left[\cdot, \cdot\right]_\cE$ and
$\left[ \cdot,  \cdot \right]_\cF$
can be represented by square matrices
and are defined through a standard assembly process
\begin{eqnarray}
  (\Ev_h)^T \mM_\cE \Dv_h
  =
  \left[\Ev_h, \Dv_h \right]_\cE
  &:=&
  \sum_{f\in \fs F}
  \left[\Ev_{h}, \Dv_{h} \right]_{\cE,f}
  =
  \sum_{f \in \fs F}
  (\Ev_{h,f})^T \mM_{\cE,f} \Dv_{h,f},
  \\
  (B_h)^T \mM_\cF M_h
  =
  \left[B_h, M_h \right]_\cF
  &:=&
  \sum_{f \in \fs F}
  \left[B_h, M_h \right]_{\cF,f}
  =
  \sum_{f \in \fs F}
  (B_{h,f})^T \mM_{\cF,f} M_{h,f},
\end{eqnarray}
where
$\left[\Ev_h, \Dv_h \right]_{\cE,f}$
and
$\left[B_h, M_h \right]_{\cF,f}$
are to be defined locally on each face $f$.

For scalar functions the definition for the local bilinear form is the simplest
\begin{equation}\label{eq:scalar mass}
  \left[B_{h}, M_{h} \right]_{\cF,f}
  :=
  |f| B_{h,f}^T M_{h,f},
  \qquad
  \text{i.e.}\quad
  \mM_{\cF,f}
  =
  |f|,
\end{equation}
where $|f|$ is the area of the face $f$.

Next, we define the discrete $\scurl_h$ operator
as a mapping from the discrete space $\cE_h$, approximating vector functions,
to the discrete space $\cF_h$, approximating scalar functions.
The definition will be made locally on each face $f$ through the identity
\begin{equation}\label{eq:def curlh}
  [\psi_h,\scurl_h \Ev_{h}]_{\cF,f}
  =
  \int_f
  \psi\, \scurl(\Ev)\ df,
\end{equation}
which must hold for any constant $\psi$ and any vector function $\Ev$ whose local discrete representation on the face $f$ is $\Ev_{h,f}$.
Identity \eqref{eq:def curlh} defines $\scurl_h$ uniquely since the RHS of \eqref{eq:def curlh}
can be expressed in terms of the vector of DoF $\Ev_{h,f}$ of $\Ev$ on the face $f$ only
using integration by parts
\begin{equation}\label{eq:int by parts for curl}
  \int_f
  \psi\, \scurl(\Ev)\ df
  =
  \psi
  \sum_{e\in \d f}\int_e
  \Ev\cdot \btau_e\ de,
\end{equation}
where $\boldsymbol\tau_e$ is a counter-clockwise tangent to the edge $e$ of the face $f$.
The edge integrals in \eqref{eq:int by parts for curl} are exactly the DoF of $\Ev$, defined in \eqref{eq:edof}, up to orientation of the edges.
On a rectangular element the local matrix $\scurl_h$ is a column vector of length four given by
\begin{align}
    \scurl_{h,f}
    =
    \frac{1}{\dx\dy}(\dx,\dy,-\dx,-\dy)^T.
\end{align}

The construction of the mass matrix $\mM_\cE$ contains many details
that are not necessary for performing dispersion reduction analysis that is the center point of this paper.
Therefore, we limit ourselves to presenting the final form of the matrix on a rectangular mesh
and refer interested readers to \cite{bokil2015dispersion} for all details.

In fact, instead of computing $\mM_{\cE,f}$ we compute an approximation to
its inverse
\begin{align}
    \mW_{\cE,f} \approx \mM_{\cE,f}^{-1},\qquad
    \mW_{\cE,f}
    =
    \frac{1}{4\dx\dy}
    \left(
      \begin{array}{rrrr}
        1+4w_1 & 4w_2 & 1-4w_1 & -4w_2 \\
        4w_2 &1+4w_3 & -4w_2 & 1-4w_3 \\
        1-4w_1 & -4w_2 &1+4w_1 & 4w_2 \\
        -4w_2 & 1-4w_3 & 4w_2 & 1+4w_3
      \end{array}
    \right).
\end{align}
Here $w_1,w_2$ and $w_3$ are free parameters.
Different values of these parameters give rise to different numerical schemes, as discussed in \cite{bokil2015dispersion}. In particular, this family of matrices $\mW_{\cE,f}$ contains the Yee-scheme as one of its members.
In order to recover the famous Yee-FDTD stencil, $\mW_\cE=\frac{1}{2\dx\dy}\mathbb I$,
one has to take $w_1=w_3=\frac{1}{4}$ and $w_2=0$.


Our global matrices are then assembled in the usual way for every face $f$.
As we are considering a second order formulation we introduce the following discrete curl-curl operator,
\begin{align}
    \mA_h
    =
    (\scurl_h)^T\mM_{\fs F} \scurl_h.
\end{align}
This matrix can be assembled from local matrices $(\scurl_{h,f})^T\dx\dy \  \scurl_{h,f}$.

\subsection{Discrete in Space Continuous in Time Formulation on Rectangular Meshes}

A discrete in space continuous in time formulation of the second order Maxwell-CP system
is as follows:\\
\emph{Find} $\bE_h\in C^2([0,T],\fs E_h\cap \mb H_0(\scurl,\Omega))$ \emph{and} $\bJ_h\in C^1([0,T],\fs E_h)$ \emph{such that:}
\begin{equation}
\label{eq:disc space cont time}
    \left\{
    \begin{array}{ll}
    \ds\ods t \bE_h + \frac{1}{\ep_0} \od t \bJ_h =-c_0^2\mW_\fs E \mA_h \bE_h, \\
    \ds \od t \bJ_h= -\w_i\bJ_h +\ep_0\w_P^2 \bE_h.
    \end{array}
    \right.
\end{equation}

\subsection{Dispersion Analysis for Discrete in Space Continuous in Time Formulation}
\label{sect:daspace}

To obtain dispersion relations we assume plane wave solutions for the electric field and the polarization current density. Thus, in the discrete in space and continuous in time formulation \eqref{eq:disc space cont time}, we represent the spatially discrete electric field, $\bE_h(t)$, and polarization current density, $\bJ_h(t)$, as
\begin{equation}\label{eq:discrete wave}
\begin{split}
    \bE_h(t) &= \cI^{\fs E_h} \left(\bE_0 e^{\i(\mb k\cdot\mb x)} \right)E(t),\\
    \
    \bJ_h(t) &= \cI^{\fs E_h} \left(\bJ_0 e^{\i(\mb k\cdot\mb x)} \right)\ J(t).
\end{split}
\end{equation}
Here $\bE_0$ and $\bJ_0$ are two-dimensional vectors
that specify the initial intensity and the orientation of the corresponding vector fields.
The continuous plane wave vector fields are characterized by the wave vector $\bk$ (whose magnitude
we refer to as the wave number $k$)
and the vectors $\bE_0$ and $\bJ_0$.
In the spatially discrete setting a similar characterization is made. We use $U_i$ to denote a DoF for either the electric field or polarization current density. The role of the vectors $\bE_0$ and $\bJ_0$ is assumed by two DoF associated with two adjacent orthogonal edges, as depicted in Figure \ref{fig:l2g_neumann}. The DoF $U_1$ is associated to a horizontal edge, while the DoF $U_2$ is associated to a vertical edge.
Other DoF, $U_i$, for the waves of the form
\eqref{eq:discrete wave}
can be written in terms of $U_1$, $U_2$, and the exponent $e^{\i(\mb{k}\cdot \mb{\dx_i})}$ as
\begin{equation}\label{eq:all dofs}
\begin{split}
  U_i = U_1 e^{\i(\mb{k}\cdot \mb{\dx_i})}
  &\qquad \text{for a horizontal edge }e_i,\\
  U_i = U_2 e^{\i(\mb{k}\cdot \mb{\dx_i})}
  &\qquad \text{for a vertical edge }e_i,
\end{split}
\end{equation}
where $\mb{\dx_i}$ is the ``shift''-vector from the center of the edge $e_1$
or $e_2$ to the center of the edge $e_i$, respectively.
In particular,
the vector of four DoF $(U_1,U_2,U_3,U_4)^T$,
corresponding to the element $f$,
can be written in terms of the vector of the first two DoF $(U_1,U_2)^T$
and $\bk=(k_x,k_y)^T$ as follows:
\begin{equation}\label{eq:def:S}
    \left(
      \begin{array}{c}
        U_1 \\
        U_2 \\
        U_3 \\
        U_4 \\
      \end{array}
    \right)
    =
    \mS
    \left(
      \begin{array}{c}
        U_1 \\
        U_2 \\
      \end{array}
    \right),
    \qquad
    \text{where}
    \quad
    \mS
    =
    \begin{pmatrix}
        1 & 0 \\
        0 & 1 \\
        \e^{\i k_y\dy} & 0 \\
        0 & e^{-\i k_x \dx}
    \end{pmatrix}.
\end{equation}

In the semi-discrete formulation \eqref{eq:disc space cont time},
we perform several multiplications of waves of the form \eqref{eq:discrete wave}
by global matrices $\mA$ and $\mW$ assembled from identical local matrices
$\mA^f$ and $\mW^f$.
The following result presents a simple way of performing such a multiplication.

\begin{lem}\label{eq:single matrix multiplication}
  Consider the result of multiplication
  \begin{equation}\label{eq:simple multiplication}
    \bV = \mZ \bU,
  \end{equation}
  where the global matrix $\mZ$ is assembled from identical local matrices $\mZ^f$
  and the vector of DoF $\bU$ has the form \eqref{eq:all dofs}
  Then the vector of DoF $\bV$ also has the form \eqref{eq:all dofs}
  The vector of two DoF $(V_1,V_2)^T$ characterizing the vector $\bV$
  depends linearly on the vector of two DoF $(U_1,U_2)^T$
  characterizing the vector $\bU$;
  the two-by-two matrix corresponding to the linear mapping
  has the form $(\mS^* \mZ^f \mS)$:
  \begin{equation}\label{eq:linear relation u12 v12}
    \left(
      \begin{array}{c}
        V_1 \\
        V_2 \\
      \end{array}
    \right)
    =
    (\mS^* \mZ^f \mS)
    \left(
      \begin{array}{c}
        U_1 \\
        U_2 \\
      \end{array}
    \right),
  \end{equation}
  where
  $\mS$ was defined in \eqref{eq:def:S} and
  $\mS^*$ is a conjugate transpose of $\mS$.
\end{lem}

\begin{proof}
  The fact that $\bV$ satisfies \eqref{eq:all dofs} follows immediately from
  the fact that $\bU$ satisfies \eqref{eq:all dofs}.
  The linear relation between $(V_1,V_2)^T$ and $(U_1,U_2)^T$
  is a direct consequence of the linear relation \eqref{eq:simple multiplication}.
  The main point of the lemma is to show that the linear relation
  \eqref{eq:linear relation u12 v12} is given by the two-by-two matrix
  $(\mS^* \mZ^f \mS)$.


  Consider the two elements $f_1$ and $f_2$ that determine the value of $V_1$ and
  the two elements $f_1$ and $f_3$ that determine the values of $V_2$,
  as shown in Figure.~\ref{fig:l2g_neumann}.
  The value of $V_1$ is a sum of the contributions from the elements
  $f_1$ and $f_2$
  \begin{equation}\label{eq:v1}
  \begin{split}
    V_1
    &=
    \left[
      \begin{array}{cccc}
        1 & 0 & 0 & 0 \\
      \end{array}
    \right]
    \mZ^f \mS
    \left[
      \begin{array}{c}
        U_1 \\
        U_2 \\
      \end{array}
    \right]
    +
    \left[
      \begin{array}{cccc}
        0 & 0 & 1 & 0 \\
      \end{array}
    \right]
    \mZ^f \mS
    \left[
      \begin{array}{c}
        e^{-\i k_2\dy}U_1 \\
        e^{-\i k_2\dy}U_2 \\
      \end{array}
    \right]
    =\\ &=
    \left[
      \begin{array}{cccc}
        1 & 0 & e^{-\i k_2\dy} & 0 \\
      \end{array}
    \right]
    \mZ^f \mS
    \left[
      \begin{array}{c}
        U_1 \\
        U_2 \\
      \end{array}
    \right].
  \end{split}
  \end{equation}
  Here the product $\mZ^f \mS (U_1,U_2)^T$ is a size four vector of contributions from the element
  to the four DoF of $\bV$, while multiplication on the left by $(1,0,0,0)$ extracts the first component of this vector.

  Similarly, to \eqref{eq:v1} we get the expression for $V_2$
  \begin{equation}\label{eq:v2}
  \begin{split}
    V_2
    &=
    \left[
      \begin{array}{cccc}
        0 & 1 & 0 & 0 \\
      \end{array}
    \right]
    \mZ^f \mS
    \left[
      \begin{array}{c}
        U_1 \\
        U_2 \\
      \end{array}
    \right]
    +
    \left[
      \begin{array}{cccc}
        0 & 0 & 0 & 1 \\
      \end{array}
    \right]
    \mZ^f \mS
    \left[
      \begin{array}{c}
        e^{-\i k_1\dx}U_1 \\
        e^{-\i k_1\dx}U_2 \\
      \end{array}
    \right]
    =\\ &=
    \left[
      \begin{array}{cccc}
        0 & 1 & 0 & e^{-\i k_1\dx} \\
      \end{array}
    \right]
    \mZ^f \mS
    \left[
      \begin{array}{c}
        U_1 \\
        U_2 \\
      \end{array}
    \right].
  \end{split}
  \end{equation}
  Combining \eqref{eq:v1} and \eqref{eq:v2} and recalling the definition
  \eqref{eq:def:S}
  of the transformation matrix $\mS$ we obtain the final result
  \eqref{eq:linear relation u12 v12}.
\end{proof}

\begin{figure}[th!]
    \centering
    \includegraphics[width=.65\textwidth]{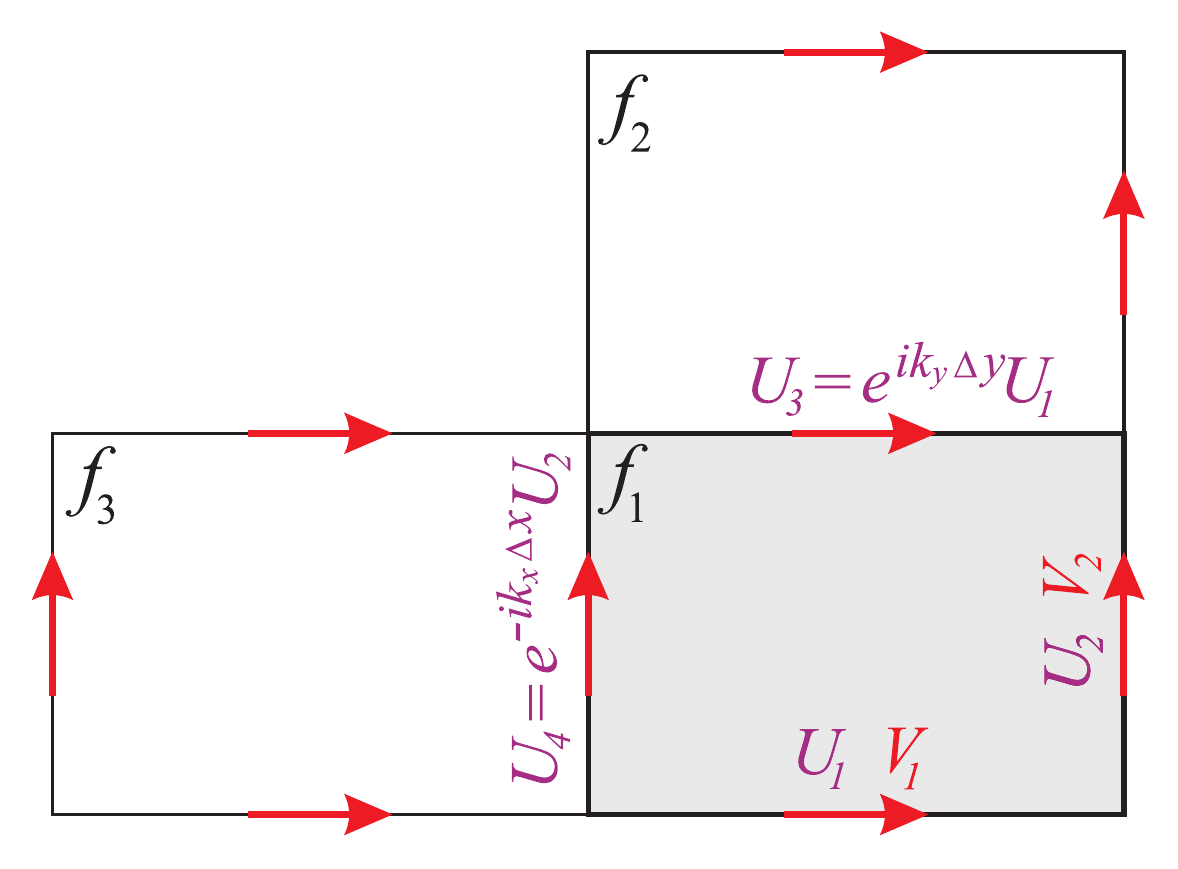}
    \caption{Three cells used to assemble the contributions after multiplication by a uniform matrix.}
    \label{fig:l2g_neumann}
\end{figure}

For a rectangular mesh the two discretization parameters are $h=\dx$, which without loss of generality we assume to be the smallest of $\dx$ and $\dy$,
and the aspect ratio of the elements
\begin{align}
    \gamma = \frac{\dy}{\dx}.
\end{align}

The characterization of waves by two DoF \eqref{eq:all dofs} and the wave vector $\mb k$
together with Lemma~\ref{eq:single matrix multiplication}
suggests rewriting the first equation in \eqref{eq:disc space cont time},
only in terms of two DoF associated to two orthogonal edges $e_1$ and $e_2$, as
\begin{align}\label{eq:two DoF evolution}
    \ods t \bE_{(e_1,e_2)} + \frac{1}{\ep_0}\od t \bJ_{(e_1,e_2)}
    &=
    -c_0^2\overline{\mW}_\fs E \overline{\mA}_h \bE_{(e_1,e_2)}.
\end{align}
In the above we use the DoF notation introduced in \eqref{eq:edof}, with $\bU_{(e_1,e_2)}: = (U_{e_1},U_{e_2})^T$, for $\bU =\bE$ or $\bU = \bJ$.
Here $\overline{\mW}_\cE$ and $\overline{\mA}_h$ are $2\times 2$
matrices defined as
\begin{align}
    \overline{\mW}_\cE =\mS^\star \mW_{\cE,f}\mS
    \qquad \text{and} \qquad
    \overline{\mA}_h=\mS^\star\mA_{h,f}\mS.
\end{align}

Performing dispersion analysis on the semi-discrete system \eqref{eq:two DoF evolution}
yields a $2\times 2$ eigenvalue problem.
The matrix $\mA_h$ has rank one, thus the product
$\overline{\mW}_\fs E\overline{\mA}_h$ is at most rank one.
This implies that one of two eigenvalues is zero.
This eigenvalue corresponds to evanescent waves,
which are not of interest as they do not propagate in space.
We focus on the non-zero eigenvalue of the matrix $\overline{\mW}_\fs E\overline{\mA}_h$.
The non-zero eigenvalue corresponds to transient waves.
Since one of the eigenvalues of this matrix is zero,
the symbol is given by the trace of the matrix as
\begin{equation}\label{eq:Fourier symbol}
\begin{split}
    \FS_h(\textbf{k})
    &=
    -c_0^2\text{Tr}(\overline{\mW}_\fs E\overline{\mA}_h)\\
    &=
    -\frac{4c_0^2}{h^2}\sin^2\frac{k_x h}{2}\left(1+(1-4w_3)\sin^2\frac{k_x h}{2}\right)\\
    &
    -\frac{32c_0^2}{\gamma h^2}w_2\sin^2\frac{k_x h}{2}\sin^2\frac{k_y \gamma h}{2}\\
    &
    -\frac{4c_0^2}{\gamma^2h^2}\sin^2\frac{k_y\gamma h}{2}\left(1+(1-4w_1)\sin^2\frac{k_y\gamma h}{2}\right).
\end{split}
\end{equation}

The difference between the continuous and the discrete symbols,
in the case of exact time integration,
defines the order of numerical dispersion.
Therefore, we are interested in making this difference, $\FS(\bk)-\FS_h(\bk)$,
as small as possible.
For a general member of the MFD family this difference is second order in $h$.
This can be seen by taking $\bk=k(\cos\theta,\sin\theta)^T$
and expanding
$\FS_h(\textbf{k})$ in a Taylor series in $h$ as
\begin{equation}\label{eq:Fourier symbol 2}
\begin{split}
    \FS_h(\textbf{k})
    =
    -(c_0 k)^2
    \Big\{1
    +\Big(\tfrac{(3w_3-1)}{3}\cos^4\theta
    &+2 \gamma w_2\cos^2\theta\sin^2\theta+ \\
    &
    +\tfrac{\gamma^2(3w_1-1)}{3}\sin^4\theta\Big) k^2 h^2
    +\mc O(h^4)\Big\}.
\end{split}
\end{equation}
We observe that we can eliminate the angular ($\theta$) dependence of the $h^2$ term
in \eqref{eq:Fourier symbol 2} through the following choice of free parameters
in the MFD scheme
\begin{align}
    \tfrac{\gamma^2(3w_1-1)}{3}
    =
    \gamma w_2
    =
    \tfrac{(3w_3-1)}{3},
    \quad \text{e.g. by taking}\qquad
    \left\{
    \begin{array}{l}
      w_1 = \frac{3 w_2 \gamma^{-1}\ \ \ +1}{3},
      \\
      w_3 = \frac{3 w_2 \gamma+1}{3}.
    \end{array}
    \right.\label{eq:paramsel1}
\end{align}
%
This yields the following discrete symbol
\begin{align}
    \FS_h(\textbf{k})
    =
    -(c_0 k)^2
    \left\{
    1+ \gamma w_2 k^2 h^2+\mc O(h^4)\right\}.
\end{align}
Taking $w_2=0$ we can eliminate the second order difference between
the continuous and discrete symbols, thus making the
numerical dispersion fourth order accurate.
Then by using a fourth order discretization in time,
we could arrive at a method with fourth order dispersion error.
This approach has its advantages especially at low spatial resolution.
Unfortunately, the storage necessary for these schemes may be prohibitive for large problems.
We will instead focus on choosing a second order time integrator so that $w_2$ can eliminate all second order dispersion errors.

\section{Exponential Time Differencing}
\label{timediscrete}
Exponential differencing is a time integration technique commonly used for lossy dielectrics.
The idea is as follows.
Consider a time dependent ODE of the form
\begin{align}\label{eq:scalar ODE}
    \dot u = c u + F(u,t).
\end{align}
To create the exact solution of \eqref{eq:scalar ODE} one exploits the fact that
\begin{equation}\label{eq:dif of product}
    \frac{d}{dt}\left(e^{-ct} u\right)
    =
    e^{-ct}\left(\dot u - c u\right).
\end{equation}
Multiply \eqref{eq:scalar ODE}
by $e^{-ct}$, integrate from $t_n$ to $t_{n+1}$ and divide by $e^{-ct_{n+1}}$
\begin{align}
    u(t_{n+1})-\e^{c\Delta t} u(t_n)
    &=
    \int_{t_n}^{t_{n+1}}
    F(u(s),s)\e^{c(\dt-s)}\ ds.
    \label{eq:exact ETD}
\end{align}
Formula \eqref{eq:exact ETD} is exact.
Thus, in principle, higher order accuracy can be obtained by using higher order discretization of the integral term, c.f. \cite{cox2002exponential}.
In practice,
due to specifics of definitions of $F(u(s),s)$, it could be convenient to use the following approximation to the integral in \eqref{eq:exact ETD}
\begin{align}
    \int_{t_n}^{t_{n+1}} F(u(s),s)\e^{c(\dt-s)}\ ds
    \approx
    F^{n+1/2} \int_{t_n}^{t_{n+1}} \e^{c(\dt-s)}\ ds
    =
    c^{-1}(\e^{c\dt}-1)F^{n+1/2},
\end{align}
where $F^{n+1/2}$ is an approximation of $F(u(t_{n+1/2}),t_{n+1/2})$.
When $F$ depends on $u$ then this approach may be implicit.
However, in our case $F$ depends only on $t$ as we employ a staggering technique so
this quantity is computed explicitly.
This approach can be generalized to vector valued ODE in time
\begin{align}\label{eq:vector ODE}
    \dot{\mb u}=\mX\mb u +\mb F(\mb u,t)
\end{align}
to produce the following discretization when $\mX$ is invertible
\begin{align}\label{eq:ETD practice}
    \mb u^{n+1}-\e^{\mX\dt} \mb u^n
    =
    \mX^{-1}(\e^{\mX \dt}-\mathbb I) \mb F^{n+1/2}.
\end{align}

\subsection{Continuous in Space Discrete in Time Formulation}

Rewrite the first order PDEs (\ref{eq:Maxwell equations}-\eqref{eq:evolutionJ}) as
\begin{align}
    \left\{
    \begin{array}{ll}
        \dot{ \mb u} = \mX \mb u+\mb F,\\
        \dot{ B} = - \scurl \mb E,
    \end{array}
\right.\label{eq:foteqs}
\end{align}
where
\begin{align}\label{eq:system pre dt}
    \mb u
    =
    \begin{pmatrix}
    \mb E\\
    \mb J
    \end{pmatrix},
    \qquad
    \mX
    =
    \begin{pmatrix}
    0 & \ds-\frac{1}{\ep_0}\\
    \ep_0\w_P^2 & -\w_i
    \end{pmatrix},
    \qquad
    \mb F
    =
    \begin{pmatrix}
    c_0^2 \ \vcurl B \\
    0
    \end{pmatrix}.
\end{align}
We will consider a time discretization
$\mb E^n$ and $\mb J^n$
of
$\mb E$ and $\mb J$ on integer time steps $t_n:=n\dt$
and $B^{n+1/2}$ of $B$
on staggered half-integer time steps $t_{n+1/2}:=(n+1/2)\dt$.

For the first equation in \eqref{eq:system pre dt} we use the ETD scheme \eqref{eq:ETD practice},
where we define the matrix
\begin{equation}
    \mathbb Y := \mX^{-1}(\e^{\mX\dt}-\mathbb I).
\end{equation}
For the second equation we use the standard time-staggered leap-frog.
Thus, the semi-discrete scheme for \eqref{eq:system pre dt} reads
\begin{align}\label{eq:first order semi}
\left\{\begin{array}{ll}
    \ds
    \begin{pmatrix}
      \mb E^{n+1} \\
      \mb J^{n+1}
    \end{pmatrix}
    =
    \e^{\mX\dt}
    \begin{pmatrix}
      \mb E^n\\
      \mb J^n
    \end{pmatrix}
    +
    \mY
    \begin{pmatrix}
    c_0^2 \ \vcurl \ B^{n+1/2}\\
    0
    \end{pmatrix},\\
    \ds B^{n+1/2} = B^{n-1/2} -\dt \ \scurl \ \mb E^{n}.
\end{array}
\right.
\end{align}

In order to obtain from \eqref{eq:first order semi} an appropriate second order formulation
we proceed similarly to how we obtained
\eqref{eq:MECP} from (\ref{eq:Maxwell equations}-\eqref{eq:evolutionJ}) by eliminating
the magnetic induction $B$ from the evolution equation.
We do this by applying a leap-frog step to both sides of the first equation
in \eqref{eq:first order semi}.
This yields a continuous in space discrete in time discretization
\begin{equation}\label{temp1}
    \begin{pmatrix}
    \mb E^{n+1}\\
    \mb J^{n+1}
    \end{pmatrix}
    =
    (\mathbb I + \\e^{\mX\dt})
    \begin{pmatrix}
    \mb E^n\\
    \mb J^n
    \end{pmatrix}
    -\e^{\mX\dt}
    \begin{pmatrix}
    \mb E^{n-1}\\
    \mb J^{n-1}
    \end{pmatrix}
    -
    c_0^2\dt\mathbb Y
    \begin{pmatrix}
    \vcurl\ \scurl\ \mb E^n\\
    0
    \end{pmatrix}.
\end{equation}

\subsection{Dispersion Analysis for Continuous in Space Discrete in Time Formulation}
To obtain a discrete in time dispersion relation we divide both sides of \eqref{temp1} by the exponential integrator $\mathbb Y$ to get
\begin{equation}
\hspace*{-12pt}\frac{1}{\dt}\mathbb Y^{-1}\left[\begin{pmatrix}\mb E^{n+1}\\\mb J^{n+1}\end{pmatrix}-(\mathbb I+e^{\mX\dt})\begin{pmatrix}\mb E^{n}\\\mb J^{n}
\end{pmatrix}+\e^{\mX \dt}\begin{pmatrix}\mb E^{n-1}\\\mb J^{n-1}\end{pmatrix}\right] =
\begin{pmatrix}-c_0^2\ \vcurl \ \scurl \ & 0\\ 0 & 0\end{pmatrix}\begin{pmatrix}\mb E^{n}\\\mb J^n\end{pmatrix}.
\end{equation}
Assuming time-harmonic solutions in the above equation we produce the system
\begin{align}
\mathbb Y^{-1}\frac{\e^{-\i\w\dt}\mathbb I-(\mathbb I+\e^{\mX \dt}) +\e^{\i\w\dt}\e^{\mX \dt} }{\dt}= \begin{pmatrix}-c_0^2\ \vcurl \ \scurl \  & 0\\ 0 & 0\end{pmatrix}.\label{eq:time_disc_disp}
\end{align}
Defining the discrete symbol in time to be
\begin{align}
\mc \FT_{\dt}(\w) = \mathbb Y^{-1}\frac{\e^{-\i\w\dt}\mathbb I-(\mathbb I+\e^{\mX \dt}) +\e^{\i\w\dt}\e^{\mX \dt} }{\dt},
\end{align}
and expanding $\FT_{\dt}$ in a Taylor Series in the variable $\dt$ we obtain
\begin{align}
\mc \FT_{\dt}(\w)=(-\w^2\mathbb I+\i\w\mX)+\frac{\dt^2}{12}(-\w^2\mathbb I+\i\w\mX)^2+\mc O(\dt^4).
\end{align}
We will make use of this expansion in our method optimization proceedure.

\section{Exponential Time Mimetic Finite Difference Method (ETMFD) for Cold Plasma}
\label{fulldiscrete}
We now present the fully discrete \textbf{Exponential Time Mimetic Finte Difference (ETMFD)} method for the Maxwell-CP model, based on MFD in space and ETD in time. Our fully discrete problem is:

\noindent Given $\mb E_h^\ell, \mb J_h^\ell\in \cE_h$ for $\ell \in \{0,1\}$, find $\mb E_h^n, \mb J_h^n\in \cE_h \quad \forall n\geq 0$ such that
\begin{equation}
\ds\begin{pmatrix}\mb E^{n+1}_h\\ \mb J^{n+1}_h\end{pmatrix}=(\mathbb I+\e^{\mathbb X\dt})\begin{pmatrix}\mb E_h^{n}\\ \mb J_h^n\end{pmatrix}-\e^{\mathbb X\dt}\begin{pmatrix}\mb E_h^{n-1}\\ \mb J^{n-1}_h\end{pmatrix} -c_0^2\dt\mathbb Y\begin{pmatrix}\mW_\cE\scurl_h^T\mA_h \mb E^n\\ 0\end{pmatrix}.\label{eq:ETMFD}
\end{equation}

\subsection{M-Adaptation of the ETMFD}
\label{madapt}
To perform m-adaptation for the ETMFD we must first find its discrete dispersion relation.
Intuitively the dispersion relation for \eqref{eq:ETMFD}
would be determined by equality between the space discrete symbol $\FS_h$ and the time discrete symbol $\FT_{\dt}$.
However, the temporal symbol as defined is matrix valued while the spatial symbol is scalar.
Consider our space discretization of the first row of \eqref{eq:MECP} assumuing $\mb E_h$ and $\mb J_h$ are appropriate
transient plane waves as discussed in \eqref{eq:Fourier symbol}
\begin{align}
\ods t \mb E_{(e_1,e_2)} + \frac{1}{\ep_0}\od t \mb J_{(e_1,e_2)} = \FS_h(\mb k_h) \mb E_{(e_1,e_2)},
\end{align}
This evolution equation implies that
$\mb E_{(e_1,e_2)}$ and $\mb J_{(e_1,e_2)}$ must be colinear as $\FS_h$ is scalar.
We therefore define the quantities
\begin{align}
E_0=|\bE_{(e_1,e_2)}|, \qquad J_0=|\bE_{(e_1,e_2)}|.
\end{align}
Thus,  the spatial symbol of the ETMFD must be some 2x2 matrix multiplied by $\FS_h(\mb k)$ acting on the vector $(E_0,J_0)^T$. However as $\FS_h(\mb k)$ does
not depend on $J_0$ and has no influence on the second row of the system \eqref{eq:time_disc_disp} so this matrix must be
\begin{align}
\mP_1=\begin{pmatrix}1&0\\ 0 & 0\end{pmatrix}.
\end{align}
Given that $\FS_h\mP_1$ is matrix valued we can now pose the discrete dispersion relation as a $2\times 2$
eigenvalue problem on the initial orientations of the fields $\mb E_0$ and $\mb J_0$,
which must be non-zero eigenvectors of the matrix $\overline{\mathbb W}_\cE \ \overline{\mathbb A}_h$.
The discrete dispersion relation for \eqref{eq:ETMFD} is then given by
\begin{align}
\FT_{\dt}(\w) \begin{pmatrix}E_0\\ J_0\end{pmatrix} = \FS_h(\mb k)\mP_1 \begin{pmatrix}E_0\\J_0\end{pmatrix}.\label{eq:etmfddisprel}
\end{align}
To perform m-adaptation we begin by choosing $w_1,w_3$ as defined in \eqref{eq:paramsel1} which eliminates dependence on angle of propagation and leaves
us with one free parameter $w_2$, i.e.
\begin{align}
    w_1 = \frac{3 w_2\gamma^{-1}+1}{3},\qquad w_3=\frac{3w_2\gamma +1}{3}.
\end{align}
In order to relate time and space discretization sizes we introduce the Courant number
\begin{align}\label{eq:def:Courant}
    \nu=\frac{c_0\dt}{h}.
\end{align}
By moving both terms in \eqref{eq:etmfddisprel} to the left side and expanding
in a Taylor series in $h$ we get
\begin{equation}\label{eq:TS in h}
\begin{split}
    0&=\Big(\FT_\dt(\w)
    -\FS_h(\textbf{k})\mathbb P_1\Big)\begin{pmatrix}E_0\\  J_0\end{pmatrix}
    = \\ &\hspace*{-10pt}=
    \left(\left(-\w^2\mathbb I +\i\w\mathbb X+c_0^2k^2\mathbb P_1\right) +\frac{h^2}{12 c_0^2}\Big(\nu^2(-\w^2\mathbb I +\i\w\mathbb X)^2+12\gamma w_2 c_0^4k^4\mathbb P_1^2\Big)\right) \begin{pmatrix}E_0\\J_0\end{pmatrix}
    +\\ &+
    \mc O(h^4)\begin{pmatrix}E_0\\J_0\end{pmatrix}.
\end{split}
\end{equation}
As $(\w,\mb k)$ is a solution of the discrete dispersion relation we have
\begin{align}
    c_0^2 k^2 \mathbb P_1
    \begin{pmatrix}
    E_0\\
    J_0
    \end{pmatrix}
    =
    \left(\w^2\mathbb I-\i \w\mathbb X+\mc O(h^2)\right)
    \begin{pmatrix}
    E_0\\
    J_0
    \end{pmatrix}.
\end{align}
Substituting this
into the order $h^2$ term in \eqref{eq:TS in h} we have
\begin{align}
    \frac{h^2}{12c_0}\left(\nu^2+12\gamma w_2\right)c_0^4k^4\mathbb P^2+\mc O(h^4).
\end{align}
We can eliminate order $h^2$ term entirely by a proper choice of the parameter $w_2$
\begin{align}\label{eq:choice w2}
    \left(\nu^2+12\gamma w_2\right)=0
    \qquad \Rightarrow \qquad
    w_2 = -\frac{\nu^2}{12\gamma}.
\end{align}
Under the choice \eqref{eq:choice w2} the dispersion error is
\begin{align}
    \Big(\FT_{\dt}(\w)-\FS_h(\mb k) \mathbb P_1\Big)
    \begin{pmatrix}E_0\\J_0\end{pmatrix}
    =
    \left(-\w^2\mathbb I+\i\w\mathbb X + c_0^2 k^2 \mathbb P_1 +\mc O(h^4)\right)
    \begin{pmatrix}E_0\\ J_0\end{pmatrix}.
\end{align}
For convenience we explicitly write out the the choice of the optimal matrix
$\mathbb W_{\fs E,f}$
\begin{align}
    \mathbb W_{\fs E,f}
    =
    \frac{1}{12 \Delta x\Delta y}
    &
    \left(\begin{array}{rrrr}
    7-\nu_y^2 & -\nu_x\nu_y & \nu_y^2-1 & \nu_x\nu_y\\
    -\nu_x\nu_y & 7-\nu_x^2 & \nu_x\nu_y  & \nu_x^2-1\\
    \nu_y^2-1 & \nu_x\nu_y & 7-\nu_y^2 & -\nu_x\nu_y\\
    \nu_x\nu_y & \nu_x^2-1 & -\nu_x\nu_y & 7-\nu_x^2
    \end{array}\right),
    \qquad
    \begin{array}{l}
      \ds\nu_x=\frac{c_0\dt}{\dx},
      \\ \\
      \ds\nu_y=\frac{c_0\dt}{\dy}.
    \end{array}
\end{align}

\section{Numerical Simulations for Specific Media}
\label{num}
For our experiments we introduce a change of variables for $\mathbb X$ which allows for an easier formulation of the matrix exponential.
\begin{align}
\mathbb X = \begin{pmatrix}0 & -\ep_0^{-1} \\ \ep_0(\alpha^2+\beta^2) & 2\alpha\end{pmatrix}, \quad \alpha = -\frac{\w_i}{2}, \ \beta=\frac{\sqrt{4 \w_P^2-\w_i^2}}{2}.
\end{align}
The ODE system governing the cold plasma model is a classical damped, driven oscillator.
For different values of the parameters the character of the system changes.
We present results for the case when the system is under damped ($\w_i^2 < 4\w_p^2$).
The matrix exponential for $\mathbb X\dt$ is given by
\begin{align}
\e^{\mathbb X \dt}  =\e^{\alpha \dt}\begin{pmatrix}\cos(\beta \dt) -\displaystyle\alpha\frac{\sin(\beta \dt)}{\beta} & -\displaystyle\frac{\sin(\beta \dt)}{\ep_0\beta}\\
\ep_0(\alpha^2+\beta^2)\displaystyle\frac{\sin(\beta \dt)}{\beta} & \cos(\beta \dt) + \alpha \displaystyle\frac{\sin(\beta \dt)}{\beta}\end{pmatrix}
  :=\begin{pmatrix}\alpha_1 & \alpha_2 \\ \beta_2 & \beta_1\end{pmatrix}.
\end{align}
The integral of this matrix is given by
\begin{align}
\int_0^{\dt}\e^{\mathbb X s}\ ds=\begin{pmatrix}\alpha_3 & \alpha_4 \\ \beta_3 & \beta_4\end{pmatrix}
,
\end{align}
where the coefficients in the matrix above are defined as
\begin{align}
\alpha_3 &:= \frac{1}{\beta}\left(\displaystyle\frac{\e^{\alpha\dt}(2\alpha\beta\cos(\beta\dt)+(\beta^2-\alpha^2)\sin(\beta\dt))-2\alpha\beta}{\alpha^2+\beta^2}\right)
,\\
\alpha_4 &:=  \frac{1}{\beta}\left(\displaystyle-\frac{\beta-\e^{\alpha\dt}(\alpha\sin(\beta\dt)-\beta\cos(\beta\dt))}{\ep_0(\alpha^2+\beta^2)}\right),\\
\beta_3 &:= \frac{1}{\beta}\left(\ep_0(\beta+\e^{\alpha\dt}(\alpha\sin(\beta\dt)-\beta\cos(\beta\dt)))\right),\\
\beta_4 &:=  \frac{1}{\beta}\left(\e^{\alpha\dt}\sin(\beta\dt)\right).
\end{align}

The second order formulation for the discrete electric field $\mathbf{E}$ and polarization current density $\mathbf{J}$, as introduced in Section \ref{madapt},
was a convenient formulation of the discrete ETMFD method for the analysis of numerical dispersion. However, in our numerical experiments we have
found that $L^2$ errors in the second order system for $\mathbf{E}$ and $\mathbf{J}$ are very sensitive to the choice of initial conditions.
Thus, for our numerical simulations we will use a different formulation of the discrete ETMFD method, with an equivalent numerical dispersion relation,
that retains the second order discrete evolution equation for the electric field, but uses a first order discrete evolution equation for the polarization current density $\mathbf{J}$.
Since the focus of this paper is on numerical dispersion optimized methods, we do not investigate the appropriate initialization of the discrete ETMFD scheme here.
We defer this investigation to future work.

The hybrid second order evolution equation for the discrete electric field $\mathbf{E}$ and first order evolution equation for the polarization density $\mathbf{J}$ is given as
\begin{align}
&\mb E^{n+1}_h = (1+\alpha_1)\mb E^n_h + \alpha_2\mb J^n_h - \alpha_1\mb E_h^{n-1} -\alpha_2\mb J_h^{n-1} -c_0^2\dt\alpha_3\mW_\cE \mA_h\mb E_h^n & n\geq 2,\\
&\mb J^{n+1}_h = \beta_1\mb J^n_h +\beta_2\mb E^n_h +\frac{\beta_3}{\alpha_3}(\mb E_h^{n+1}-\alpha_1\mb E_h^n-\alpha_2\mb J_h^n) &n\geq 1.
\end{align}
This formulation is explicit when we compute $\mb E^{n+1}$ before $\mb J^{n+1}$. It requires three initial conditions given by
\begin{align}
&\mb E^0_h = \mc I^{\cE_h}(\mb E(0)),\qquad \mb E^1_h=\mc I^{\cE_h}(\mb E(\dt)), \qquad \mb J^0_h=\mc I^{\cE_h}(\mb J(0)).
\end{align}
In our numerical simulations we used a midpoint quadrature on every edge for $\mb E_h$ and computed $\mb J_h$ exactly; i.e.,
\begin{align}
\mb E_{h|e}^j = \boldsymbol\tau_e \cdot \mb E(x_c,y_c,j\dt), \ \ j=\{0,1\}, \qquad \mb J_{h|e}^0 = \frac{1}{|e|}\int_e \mb J(x,y,0)\cdot\boldsymbol\tau_e ds.\label{eq:icquad}
\end{align}
\begin{exper}
In our first experiment we investigate the numerical anisotropy of our method.
If $(\w_n,\mb k_n)$ are solutions of the numerical dispersion relation, then as eigenvalue pairs they satisfy the relation
\begin{equation}
\label{eq:numdisp}
\text{det}\left(\FT_{\dt}(\w_n)-\FS_h(\mb k_n)\mP_1\right)=0.
\end{equation}
 The continuous dispersion relation between $\omega$ and $\textbf{k}$ can be written as
\begin{equation}
-\i\w^3+\w_i\w^2-\i(\w_p^2+c_0^2k^2)\w-\w_i c_0^2k^2=0\label{eq:cipexdisp}.
\end{equation}
Assume that $\mb k$ is fixed and real valued. Let $(\w,\mb k)$ be a solution to \eqref{eq:cipexdisp}. We define the relative dispersion error as
\begin{equation}
\mc E(\w) = \frac{1}{|\w|}\text{det}\left(\FT_{\dt}(\w)-\FS_h(\mb k)\mP_1\right )
\end{equation}
which is analogous to a local truncation error, i.e., we measure how close the true root $\omega$ of the continuous dispersion relation \eqref{eq:cipexdisp} is to being a root of the numerical dispersion relation \eqref{eq:numdisp}.
We parameterize the wave vector as $\mb k =4(\cos\theta,\sin\theta)$.
For this experiment we will also choose $\w_i=1$ and $\w_p=1$. In Figure \ref{fig:aniso} we plot $\mc E(\w)$ as a polar function of $\theta$ on a log scale.
A perfect circe in this diagram indicates isotropic error, otherwise the error shows the directional dependence of dispersion.
For the purpose of comparison we also include the relative dispersion error of ETD in time and Yee like staggering in space, which we refer to as the ET-Yee scheme.
\begin{figure}[th!]
\centering
\begin{tabular}{cc}
\includegraphics[width=.5\textwidth]{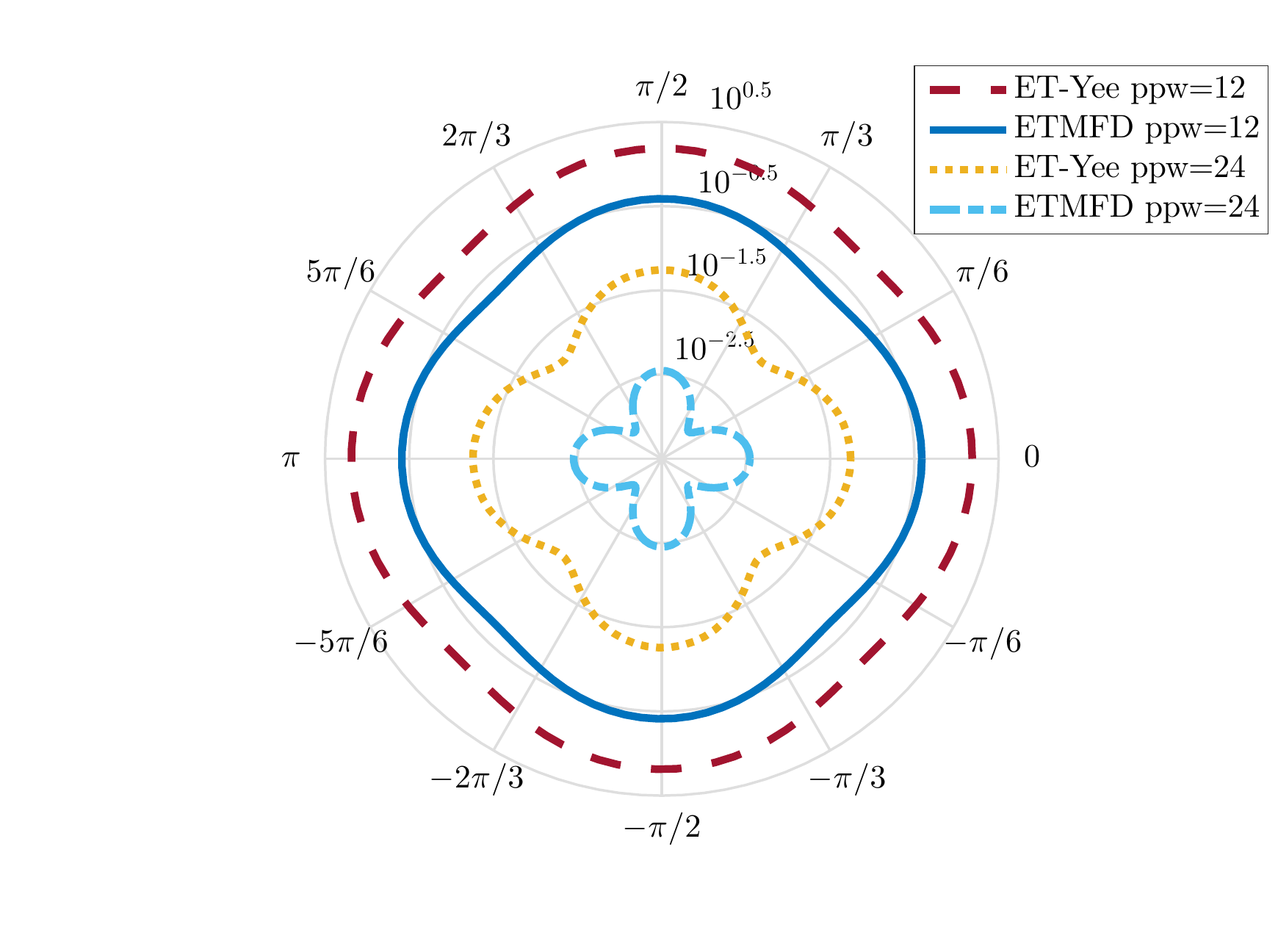} &\includegraphics[width=.5\textwidth]{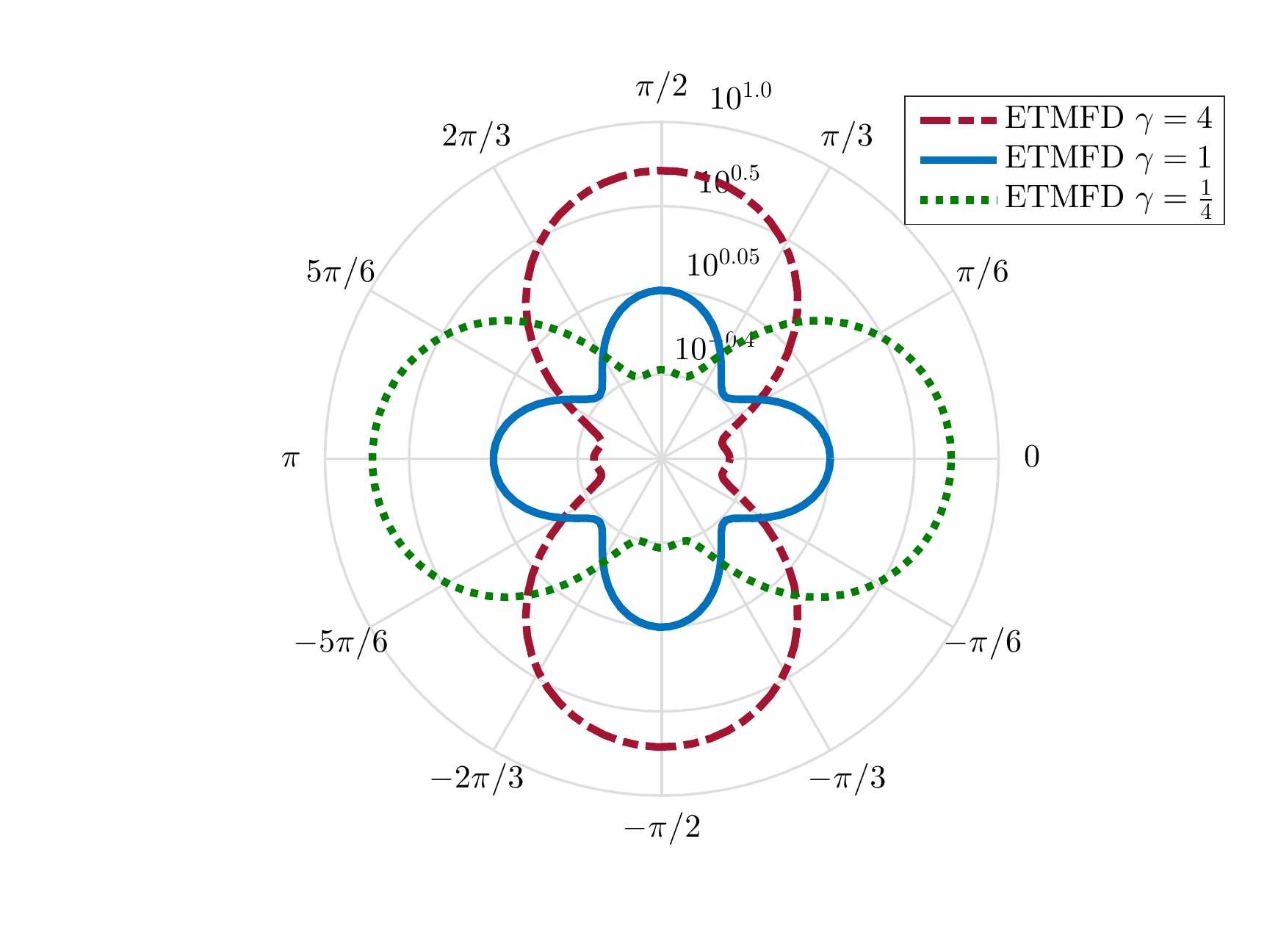}\\
(a) & (b)
\end{tabular}
\caption{We consider a cold isotropic plasma with $\w_P=1$ and $\w_i=1$. Figure (a) plots relative dispersion error for a wave with $k=4$
 and resolved at 12 and 24 points per wavelength on a mesh with an aspect ratio $\gamma=1$ for both the ET-Yee and ETMFD schemes. We choose the Courant number to be $\nu=\frac{1}{2}$.  Figure (b) plots the relative dispersion error for a cell with $\dx\dy=12^{-2}$
 for the aspect ratios $\gamma=4,1,\frac{1}{4}$. In this case we choose the Courant number to be $\nu =\frac{1}{2}\min\{\gamma^3,1\}$.}\label{fig:aniso}
\end{figure}
Figure \ref{fig:aniso}  illustrates that while the dispersion error of our method is anisotropic it is significantly reduced compared to that of the ET-Yee scheme.
By varying the aspect ratio we find that we can reduce dispersion error in the direction of increased refinement at the expense of increased dispersion error in the less refined direction.
\end{exper}

\begin{exper}
In our second experiment we will investigate the accuracy of our ETMFD method for discretizing problems with a known exact solution.
For $\mb k=(k_x,k_y)^T$ with $k_x,k_y\in\pi\mathbb Z$, let $a+\i b=\w$ be a (complex) root of the dispersion relation \eqref{eq:cipexdisp}.
We consider the exact solution for the Maxwell-CP model given by
\begin{align}
\mb E(x,y,t) &= \e^{a t}\cos(b t)\begin{pmatrix}-k_y \cos(k_x x)\sin(k_y y)\\ k_x \sin(k_x x)\cos(k_y y)\end{pmatrix},\\
\mb J(x,y,t) &= \ep_0\w_p^2\e^{a t} \frac{(a+\w_i)\cos(b t) + b\sin(b t)}{b^2+(a+\w_i)^2}\begin{pmatrix}-k_y \cos(k_x x)\sin(k_y y)\\ k_x \sin(k_x x)\cos(k_y y)\end{pmatrix}.
\end{align}
For our experiments we consider $\w_P=\w_i=\ep_0=c=1$ and $k_x=k_y=\pi$. For this we have $a\approx0.023$ and $b\approx 4.55$.
We choose the final time to be $T=4$. To calculate relative $L^2$ errors we use an appropriate inner product, based on our mimetic discretization, which is defined as
\begin{align}
\mc E^h_{L^2}(\mb F_h^n) := \frac{\sqrt{(\mb F_h^n - \cI^{\cE_h}(\mb F(n\dt))^T\mM_{\cE} (\mb F_h^n - \cI^{\cE_h}(\mb F(n\dt)) }}{\sqrt{\cI^{\cE_h}(\mb F(n\dt))^T\mathbb M_\fs E \cI^{\cE_h}(\mb F(n\dt))}},
\end{align}
where $\mb F_h^n=(\mb E_h^n,\mb J_h^n)^T$ and the interpolation $\mc I^{\fs E_h}$ operator is defined in \eqref{eq:edof}.

To define the dispersion error we fit an appropriate temporal function, $F(t:\w_h)$, to temporal grid data $\{\mb E_{h,e_i}^n\}_{n=0}^N$ at some edge $e_i$
to find the best discrete frequency $w_h$. To calculate the relative dispersion errors, we perform the following procedure. If $(a_h,b_h)$ is the result of
the non-linear least squares fitting of time tracking data $\{\mb F^n_{h|e}\}_{n=1}^N$ to the appropriate function $(\exp(a_h t)\cos(b_h t)$ for the electric field and
$\ep_0\w_p^2\e^{a_h t} \frac{(a_h+\w_i)\cos(b_h t) + b_h\sin(b_h t)}{b_h^2+(a_h+\w_i)^2}$ for the current density) then we define the relative dispersion error by
\begin{align}
\mc E_d^h(\mb F_h) := \sqrt{\frac{(a-a_h)^2+(b-b_h)^2}{a^2+b^2}}
\end{align}
where $a,b$ are the true data. For comparison, we have also performed our simulations with the corresponding ET-Yee scheme (i.e., Yee spatial staggering with ETD),
which is second order accurate in space and time. In Table \ref{tab:l2r} we present relative $L^2$ errors in the electric field and polarization density,
while in Table \ref{tab:dr} we present relative dispersion errors for the electric field and polarization density, respectively.
Figures \ref{fig:l2rs}, and \ref{fig:drs} plot the results of Tables \ref{tab:l2r}-\ref{tab:dr}. Our results indicate fourth order dispersion and $L^2$ error
convergence for the ETMFD as compared to the corresponding (well known) second order convergence for the ET-Yee scheme.
\begin{figure}[th!]
\centering
\includegraphics[width=\textwidth]{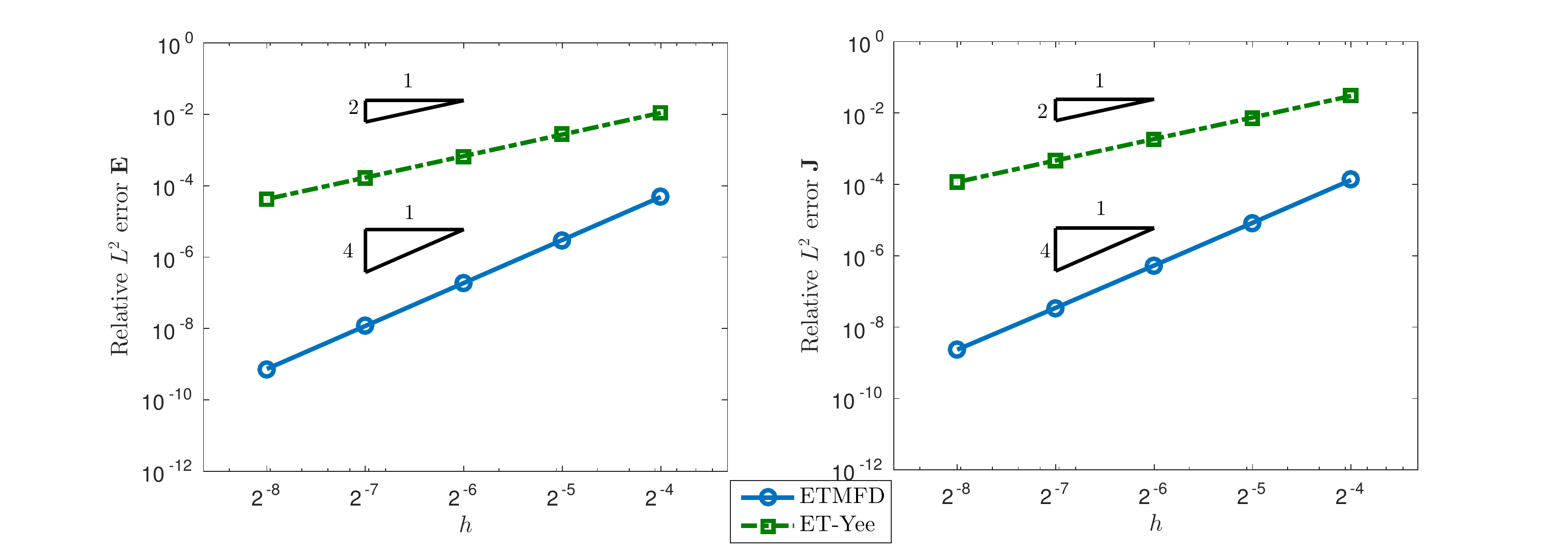}
\caption{Relative $L^2$ errors for Experiment 2.}\label{fig:l2rs}
\includegraphics[width=\textwidth]{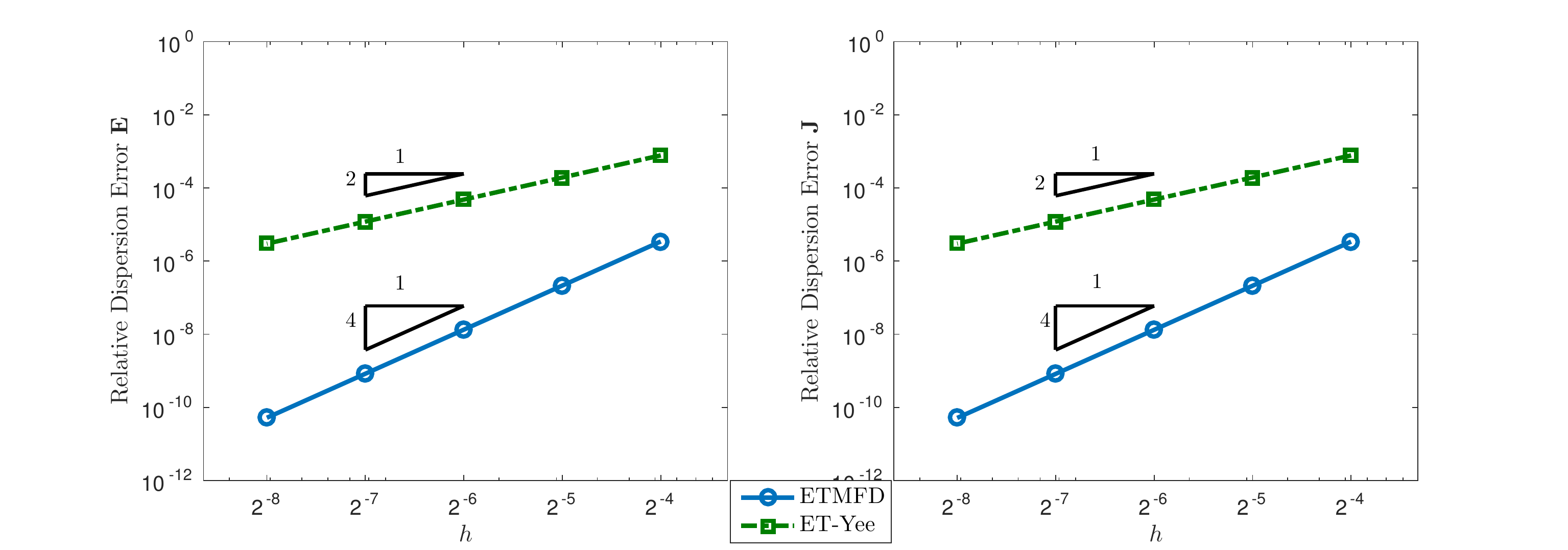}
\caption{Relative dispersion errors for Experiment 2.}\label{fig:drs}
\end{figure}

\begin{table}[th!]
\centering
\caption{Relative $L^2$ Errors for Experiment 2.}\label{tab:l2r}
{\small
\begin{tabular}{c|c|c|c|c|c|c|c|c|}
\cline{2-9}&\multicolumn{4}{|c|}{Electric Field, $\mb E$} &\multicolumn{4}{|c|}{Current Density, $\mb J$}\\
\hline
\multicolumn{1}{|c|}{$\log_2(h)$} & ET-Yee & rate &ETMFD &rate & ET-Yee & rate & ETMFD & rate\\
\hline
\hline
\multicolumn{1}{|c|}{-4}&1.1024e-02     &         &4.8495e-05    &      & 3.0064e-02   &          &1.3322e-04    &      \\
\multicolumn{1}{|c|}{-5}&2.7237e-03     &2.0170   &3.0206e-06    &4.0049& 7.4940e-03   & 2.0042   &8.3901e-06    &3.9890\\
\multicolumn{1}{|c|}{-6}&6.7826e-04     &2.0057   &1.8844e-07    &4.0026& 1.8704e-03   & 2.0024   &5.3485e-07    &3.9715\\
\multicolumn{1}{|c|}{-7}&1.6931e-04     &2.0021   &1.1767e-08    &4.0013& 4.6717e-04   & 2.0013   &3.4784e-08    &3.9426\\
\multicolumn{1}{|c|}{-8}&4.2303e-05     &2.0009   &7.3501e-10    &4.0008& 1.1674e-04   & 2.0007   &2.3361e-09    &3.8963\\
\hline
\end{tabular}}
\end{table}

\begin{table}[th!]
\centering
\caption{Relative Dispersion Errors for Experiment 2.}\label{tab:dr}
{\small
\begin{tabular}{c|c|c|c|c|c|c|c|c|}
\cline{2-9}&\multicolumn{4}{|c|}{Electric Field, $\mb E$} &\multicolumn{4}{|c|}{Current Density, $\mb J$}\\
\hline
\multicolumn{1}{|c|}{$\log_2(h)$} & ET-Yee & rate &ETMFD &rate & ET-Yee & rate & ETMFD & rate\\
\hline
\hline
\multicolumn{1}{|c|}{-4}&7.7638e-04 &        &3.4427e-06 &        &8.7152e-04  &         &3.4530e-06  & \\
\multicolumn{1}{|c|}{-5}&1.9280e-04 &2.0129  &2.1407e-07 &4.0080  &2.1720e-04  &2.0045   &2.1487e-07  &4.0063\\
\multicolumn{1}{|c|}{-6}&4.8070e-05 &2.0066  &1.3345e-08 &4.0042  &5.4246e-05  &2.0014   &1.3399e-08  &4.0032\\
\multicolumn{1}{|c|}{-7}&1.2002e-05 &2.0033  &8.3287e-10 &4.0021  &1.3557e-05  &2.0005   &8.3655e-10  &4.0016\\
\multicolumn{1}{|c|}{-8}&2.9985e-06 &2.0017  &5.1994e-11 &3.9892  &3.3886e-06  &2.0003   &5.2097e-11  &4.0052\\
\hline
\end{tabular}}
\end{table}

\end{exper}

\section{Conclusions}
\label{conclude}

We have constructed a new successful m-adaptation of Mimetic Finite Difference (MFD)
method for Maxwell's equations in a cold plasma.
We started from a second order edge based family of MFD discretizations in space.
We used a generalized form of mass lumping.
This was done, on one hand, to obtain a fully explicit scheme, thus avoiding linear solves at every time step.
As a result the new scheme is highly efficient.
On the other hand, the generalized form of mass lumping preserves free parameters in the MFD discretization.
This allows for further optimization within the family.

We demonstrated that using the standard leapfrog time stepping does not allow reduce the numerical dispersion within the MFD family.
Fortunately, using exponential time differencing allows for a successful m-adaptation of the MFD family.
For the optimal choice of parameters in the exponential time MFD (ET-MFD) discretization
the numerical dispersion errors were shown to be fourth order
as opposed to second order for a general member of the ET-MFD family.
Numerical simulations independently verified our theoretical results
showing fourth order numerical dispersion and $L^2$-errors for some special solutions.

One of the advantages of our m-adapted ET-MFD method over other fourth order methods, that have been constructed in the literature using the modified equation approach (see e.g. \cite{young4}),
is smaller stencil size as compared to those of other fourth order methods.
This is due to the low order base discredization.
Higher order approximation is a result m-adaptation procedure, which is possible due
to mesh regularity and symmetry.

The use of ETD offers a number of advantages in complex dispersive media.
First, it allows for an explicit staggering of the electric field and current density from the magnetic field.
This is in contrast to time averaging schemes which are semi-implicit with spatial staggering.
Though ETD may require smaller time steps than a semi-implicit approach, linear dispersive media such
as a cold isotropic plasma are stiff media requiring very small time steps in numerical discretizations in order to
capture the fast decaying transients in the media.
Thus a cheaper explicit scheme is preferable to a more expensive implicit scheme when run with comparable time steps.

Finally, our approach can be interpreted as a generalization (though non-trivial)
of m-adaptation in vacuum.
It inherits a similar structure for its discrete symbol in time and allows for successful optimization over numerical dispersion errors.
As a consequence, the optimal choice of free MFD parameters and the corresponding local mass matrices turn out to be the same in the case of vacuum and Maxwell-CP model.
In the future we will investigate the sensitivity of errors to initial conditions and
extend the ideas and techniques considered here to other types of linear and nonlinear dispersive materials.

Maxwell's equations include divergence constraints on the electric and magnetic flux densities.
It is well known that if solutions satisfy these divergence constraints initially then the curl
 equations guarantee that these conditions are satisfied at later times.
Thus, the divergence constraints are redundant in the continuum equations as long as they are satisfied by initial conditions.
 This property may or may not hold for solutions to discrete approximations of Maxwell's equations.
We are currently working on understanding the divergence properties of our ETMFD method along with a stability analysis of our schemes.
These and related issues will be part of a future paper that will address the construction of ETMFD schemes for a large class of linear dispersive models.

\section{Acknowledgments}

V.~Gyrya's work was carried out under the auspices of the National Nuclear Security Administration of the U.S.~Department of Energy at Los Alamos National Laboratory under Computing Research (ASCR) Program in Applied Mathematics Research.

D.~A.~McGregor's work is supported by the U.S. Department of Energy National Energy Technology Laboratory Graduate Professional Internship Program.

\bibliographystyle{plain}

\appendix
\section{M-Adaptation for leapfrog time stepping}
\label{appndx}
We demonstrate the need for exponential time differencing in the spatially discretized cold plasma model to produce a fully discrete method with
high order numerical dispersion by first considering the case of a simple conductive medium. A conductive medium is a special case of
CP where we take $\w_i\to\infty$ and assume the ratio $\ep\w_P^2/\w_i\to\sigma$ as $\w_i\to\infty$.
A conductive medium is modeled by the second order PDE
\begin{equation}
\displaystyle\ods t \mb E +\frac{\sigma}{\ep_0} \displaystyle \od t \mb E = -c_0^2\vcurl\ \scurl\mb E,\\
\end{equation}
in which $\sigma$ is the electrical conductivity. Let $\tau=\frac{\ep_0}{\sigma}$.
The standard Leapfrog discretization in time with semi-implicit time averaging of the low order term gives us the scheme
\begin{align}
\frac{\mb E^{n+1}-2\mb E^n+\mb E^{n-1}}{\dt^2}+\frac{\mb E^{n+1}-\mb E^{n-1}}{2\tau\dt}=-c_0^2\vcurl\ \scurl\mb E^{n}.
\end{align}
The symbol of the time discretization and its expansion in $\w$ is given by
\begin{align}
\frac{-4\sin^2\frac{\w\dt}{2}}{\dt^2}-\frac{\i}{\tau}\frac{\sin\w\dt}{\dt}=-\w^2-\frac{\i}{\tau}\w +\frac{\dt^2}{12}\w^2\left(\w^2+\frac{2\i}{\tau}\w\right)+\mc O(\dt^4).
\end{align}
Defining the Courant number to be $\nu=\frac{c\dx}{\dt}$, and discretizing in space using the MFD for a rectangular mesh gives us a discrete dispersion relation of the form
\begin{align}
0&=\frac{-4\sin^2\frac{\w\dt}{2}}{\dt^2}-\frac{\i}{\tau}\frac{\sin\w\dt}{\dt}-\FS_h(\bk) \\&= -\w^2-\frac{\i}{\tau}\w+c_0^2k^2 +\frac{h^2}{12 c_0^2}
\left(\nu^2\w^2\left(\w^2+\frac{2\i}{\tau}\w\right)+12\gamma w_2c_0^4k^4\right)+\mc O(h^4).
\end{align}
As the whole series must be equal to zero, we have that
\begin{align}
c_0^2k^2=\w^2+\frac{\i}{\tau}\w + \mc O(h^2).
\end{align}
Substituting this into the $\dx^2$ term of the Taylor series we arrive at
\begin{align}
0 = \frac{h^2}{12c_0^2}\left(\nu^2\w^2\left(\w^2+\frac{2\i}{\tau}\w\right)-12\gamma w_2\left(\w^2+\frac{\i}{\tau}\w\right)^2\right)+\mc O(h^4).
\end{align}
As
\begin{align}
\w^2\left(\w^2+\frac{2\i}{\tau}\w\right)\neq\left(\w^2+\frac{\i}{\tau}\w\right)^2, \forall \w,
\end{align}
we cannot choose $w_2$ independent of $\w$ to eliminate this term in the dispersion error.
We consider this a failure as we are interested in time domain schemes that have higher order numerical dispersion error rather than schemes
that reduce or eliminate numerical dispersion at a design frequency (though such schemes are of interest in their own right).

\end{document}